\documentclass[11pt]{article}
\usepackage{mylatexsetting}
\newcommand{\Nzero}{\mathbb N_0}
\newcommand{\HypF}[2]{{}_{#1}F_{#2}}
\newcommand{\E}{\mathbb E}
\newcommand{\Var}{\operatorname{Var}}
\newcommand{\Cov}{\operatorname{Cov}}

\providecommand{\R}{\mathbb R}
\newcommand{\Pone}{P_{\theta_1}}
\newcommand{\Ptwo}{P_{\theta_2}}
\theoremstyle{definition}
\newtheorem{assumption}[thm]{Assumption}

\def\mytitle{A Probabilistnoic Sign Rule for Quotients of Positive Series and Integral Transforms}

\def\myabstract{%
This paper develops a probabilistic sign rule for quotients of functions represented by positive series or integrals. For a function in this class, normalising the summand function in the series case or the integrand function in the integral case induces a probability law under which parameter log-derivatives of the function are expressed as moments of kernels, the log-derivatives of the same summand or integrand function with respect to the same parameters. The resulting moment identities reduce quotient monotonicity, log-supermodularity, and log-convexity to sign criteria based on kernel monotonicity, stochastic ordering of the induced laws, and covariance or variance identities. The criteria are applied to generalised hypergeometric, Stieltjes-transform, and Prabhakar quotients, yielding new Tur\'an inequalities, two-sided Stieltjes bounds, and a local failure threshold for a monotonicity conjecture for the zero-balanced Gauss function.
}

\title{\mytitle}
\author{Zakaria Derbazi\footnote{Correspondence address: z.derbazi@qmul.ac.uk}\\{\it\small Queen Mary University of London}}
\date{}

\begin{document}

\maketitle

\begin{abstract}
\myabstract
\end{abstract}

\medskip\noindent \textbf{MSC 2020:} Primary 33C20, 26D07; Secondary 60E15

\medskip\noindent \textbf{Keywords:} \begin{minipage}[t]{0.84\textwidth}
special functions, quotient monotonicity, stochastic ordering, log-convexity, Tur{\'a}n-type inequalities%
\end{minipage}

\section{Introduction}
\label{sec:introduction}


The question motivating this paper arose in applied probability, where the monotone case of optimal stopping in two last-success problems reduces to determining the sign of a parameter log-derivative of a Gaussian hypergeometric function in one instance~\cite{GD2} and of a Kummer function in another~\cite{D2}.  Direct manipulation of these special functions is delicate, and the two cases called for distinct arguments. This highlights the need for a unified approach  that applies broadly across families.

To this end, we identify a common mechanism for functions represented by positive series or integrals. To fix ideas, take the Kummer function
\[
        M(a;c;x)=\sum_{n\ge0}\frac{(a)_n}{(c)_n}\frac{x^n}{n!},
        \qquad a,c,x>0.
\]
Fix a number \(\delta > 0\) and consider the monotonicity of the maps
\[
        x\mapsto\frac{M(a;c + \delta;x)}{M(a;c;x)},
        \qquad
        a\mapsto\frac{M(a;c + \delta;x)}{M(a;c;x)},
        \qquad
        c\mapsto\frac{M(a;c + \delta;x)}{M(a;c;x)}.
\]
Each quotient is positive, so its monotonicity is equivalently determined by the signs of the log-derivatives
\begin{equation}
\label{eq:intro-three-logderivs}
        \partial_x\log\frac{M(a;c+\delta;x)}{M(a;c;x)},
        \qquad
        \partial_a\log\frac{M(a;c+\delta;x)}{M(a;c;x)},
        \qquad
        \partial_c\log\frac{M(a;c+\delta;x)}{M(a;c;x)}.
\end{equation}
For each fixed \(a,c,x\), the terms \((a)_n x^n/((c)_n n!)\), normalised by their sum \(M(a;c;x)\), define a probability law \(P_{2}\) on the nonnegative integers, the law of the denominator.  Replacing \(c\) by  \(c+\delta\) gives a second law \(P_{1}\), the law of the numerator. 

Each argument or parameter of \(M\) has an associated {\it kernel}, the log-derivative of the \(n\)th term with respect to that variable. For the first two derivatives in~\eqref{eq:intro-three-logderivs}, the differentiated variable is not the shifted parameter \(c\), so the derivative contrasts the mean of the same kernel under \(P_{1}\) and \(P_{2}\). A stochastic ordering of the two laws then yields the sign of this difference. The third derivative is a one-parameter curvature question, settled when the kernel with respect to \(c\) is nondecreasing in \(c\).

These are instances of a kernel-based framework that we formalise as follows. Let $\Phi(\lambda,\theta)$ represent a series or integral with positive terms, with {\it monotonicity} variable $\lambda$ and {\it shifted} parameter $\theta$. 
Denote the $\lambda$-kernel and $\theta$-kernel  by  $K_\lambda$ and $L_\theta$, and let $\ell$ be the likelihood ratio of the two laws $P_{\theta_1}$ and $P_{\theta_2}$, induced by two distinct values $\theta_1,\theta_2$ of the shifted parameter.  Write $\varepsilon_K, \varepsilon_L, \varepsilon_\ell$ for the monotonicity signs of $K_\lambda, L_\theta$ and $\ell$, respectively, where this sign is $1$ when the function is nondecreasing, $-1$ when nonincreasing, and $0$ otherwise.

The parameter log-derivatives of $\Phi$ are moments of these kernels under the induced law(s), and each moment identity has a corresponding sign interpretation. The expectation of $K_\lambda$ governs the monotonicity in $\lambda$ of the quotient $\Phi(\lambda,\theta_1)/\Phi(\lambda,\theta_2)$, determined by a stochastic ordering of $P_{\theta_1}$ and $P_{\theta_2}$ when $K_\lambda$ is monotone. The covariance of the two kernels $K_\lambda, L_\theta$ under one law is the mixed second derivative \(\partial_\theta\partial_\lambda\log\Phi\), and therefore controls log-supermodularity and its associated determinantal inequality. These are the one-kernel and two-kernel forms of the {\it Sign Rule}
\[
        \varepsilon_K\,\varepsilon_\ell\, \partial_\lambda \log\frac{\Phi(\lambda,\theta_1)}{\Phi(\lambda,\theta_2)} \ge  0,
        \quad
        \varepsilon_K\,\varepsilon_L\, \partial_\theta\partial_\lambda\log\Phi(\lambda,\theta)\ge 0,
\]
for admissible parameters $\theta_1, \theta_2$ and $\lambda$. These inequalities are the sign criteria obtained from kernel monotonicity and, in the one-kernel case, stochastic ordering of the induced laws. A one-parameter curvature version, developed in the sequel, provides sufficient conditions for determining log-convexity and deriving the corresponding Tur\'an-type inequalities.

The literature on quotient monotonicity and related Tur\'an-type inequalities for functional series and integral transforms appears under two main strands.  The first is analytical and comprises the lemma of Biernacki and Krzyż~\cite{BiernackiKrzyz1955}, with extensions to functional series and integral transforms~\cite{YangChuWang2015} and to Mittag-Leffler quotients~\cite{MaoTian2024PAMS}.  The second strand, probabilistic and less developed, uses moments of an auxiliary distribution to obtain complete-monotonicity and Tur\'an-type inequalities~\cite{TomovskiMehrez2017}.

The present work belongs to the second type, but rather than computing moments directly, we decide the sign of a parameter derivative by a stochastic ordering of the induced laws, read from the kernels.  Stochastic order is thus the input that proves monotonicity, rather than an output generated from it.  This stochastic order argument recovers the Biernacki--Krzyż lemma and its extensions, while the variance and covariance criteria recover the moment computations of the second strand. For a detailed coverage of the kernel method and stochastic order for probability laws generated by special functions see~\cite{Derbazi2026kernel}.

The paper is organised as follows. Section~\ref{sec:framework} develops the framework by deriving the moment identities behind the sign criteria: the expectation of one kernel for quotient monotonicity, the covariance of two kernels for mixed derivatives and associated determinantal inequalities, and the one-parameter curvature identity for log-convexity. Section~\ref{sec:sample-classical} applies these to the generalised hypergeometric function, the generalised Stieltjes transform, and the Prabhakar function, and then to three named problems: a remainder Tur\'an problem of Mehrez and Sitnik for the Prabhakar function, with a sharp uniform range in the case \(q=1\), the local failure threshold in the zero-balanced Gauss monotonicity of Baricz, and a positive-transform criterion for shifted-hypergeometric Tur\'an inequalities of Kalmykov--Karp type. Section~\ref{sec:conclusion-final} concludes by highlighting the limits of this method and future extensions.

\section{The probabilistic framework}
\label{sec:framework}


Throughout the paper,  all random variables are defined on a complete probability space, fixed once and for all. The notation \(X\sim P\) means that $P$ is the probability law of  $X$.  Expectations, variances and covariances are denoted, respectively, by the symbols \(\E,\, \Var,\, \Cov\), without subscripts, and are taken under the law of the random variable. The symbol \(\sgn\) designates the sign function, taking value \(1\) when its argument is positive, \(-1\) when negative, and \(0\) otherwise. Finally, given a real function $f$, its monotonicity direction is denoted $\varepsilon_f$, where $\varepsilon_f=+1$ when $f$ is nondecreasing, $\varepsilon_f=-1$ when $f$ is nonincreasing, and $\varepsilon_f=0$ otherwise.

\subsection{Preliminaries}

\label{sub.preliminaries}
Let \(\Lambda\subseteq\mathbb{R}\) be an interval, \(\Theta\subseteq\mathbb{R}\) an ordered parameter set, and \(J\) an ordered support, either an interval of \(\R\) or a subset of \(\Nzero\). The support carries Lebesgue measure in the first case and the counting measure in the second, denoted \(\mu\) throughout.

To represent the summands or integrands, we introduce a positive \emph{weight} function, denoted \(w_{\lambda,\theta}:J\to(0,\infty)\),  with monotonicity parameter \(\lambda\in\Lambda\) and shifted parameter \(\theta\in\Theta\). Any further fixed parameters or arguments of the special function are suppressed from the notation and absorbed into \(w_{\lambda,\theta}\). The log-derivative of the weight function in a parameter yields the \(\lambda\)-{kernel} and the \(\theta\)-kernel
\begin{equation}
\label{def.kernel}
        K_{\lambda}(t)=\partial_\lambda\log w_{\lambda,\theta}(t),
        \qquad
        L_{\theta}(t)=\partial_\theta\log w_{\lambda,\theta}(t).
\end{equation}
Summing or integrating the weight function produces the function \(\Phi:\Lambda\times\Theta\to(0,\infty)\), given by
\begin{equation}
\label{def.phi}
\Phi(\lambda,\theta)=\int_J w_{\lambda,\theta}(t)\,\dx[\mu(t)],
\end{equation}
when $J \subseteq \mr$ and by  the series
\begin{equation*}
\Phi(\lambda,\theta)=\sum_{n\in J} w_{\lambda,\theta}(n),
\end{equation*}
when \(J\subseteq\mathbb{N}_0\).
For each \(\lambda \in \Lambda\), let \(P_\theta\) denote the probability law on \(J\) whose density with respect to \(\mu\) is \(w_{\lambda,\theta}/\Phi(\lambda,\theta)\):
\begin{equation}
\label{def.law}
P_{\theta}(\dx[t])=\frac{w_{\lambda,\theta}(t)}{\Phi(\lambda,\theta)}\,\dx[\mu(t)].
\end{equation}
Fix $\theta_1, \theta_2 \in \Theta$ and define the weight functions ratio
\begin{equation}
\label{def.weight_ratio}
\ell_{\theta_1, \theta_2}(t)=w_{\lambda,\theta_1}(t)/w_{\lambda,\theta_2}(t).
\end{equation}
This object is linked to the likelihood ratio, defined for two induced laws $P_{\theta_1}, P_{\theta_2}$, by
\begin{equation}
        \label{def.density_ratio}
        \frac{\dx[\Pone]}{\dx[\Ptwo]}(t) =  \frac{\Phi(\lambda,\theta_2)}{\Phi(\lambda,\theta_1)} \ell_{\theta_1, \theta_2}(t),
\end{equation}
where $\Phi(\lambda, \theta_2) / \Phi(\lambda, \theta_1) >0$ is a constant independent of $t$.

When $\theta$ coincides with $\lambda$, we drop the subscript $\theta$ from $w$ and the second argument of $\Phi$, writing $w_\lambda$ and $\Phi(\lambda)$. The induced laws and weight ratios are then carried with the subscript $\lambda$.
\begin{define}[Stochastic orders~\cite{ShakedShanthikumar}]
        \label{def:orders}
Let  $P_{\theta_1}$ and  $P_{\theta_2}$ be two probability measures on an ordered support \(J\).
\begin{itemize}

\item \emph{Usual stochastic order}: $P_{\theta_2}$ is said to be smaller than $P_{\theta_1}$ in the usual stochastic order, written \(P_{\theta_2} \st P_{\theta_1}\), if $P_{\theta_2}([t, \infty) \cap J) \le P_{\theta_1}([t, \infty) \cap J)$  for all $t \in J$. Equivalently, with \(X_i\sim P_{\theta_i}\), \(\E\,h(X_2)\le\E\,h(X_1)\) for every nondecreasing \(h\) for which both expectations exist.

\item \emph{Likelihood-ratio order}: $P_{\theta_2}$ is said to be smaller than $P_{\theta_1}$ in the likelihood-ratio order, written \(P_{\theta_2} \lr P_{\theta_1}\), if  the likelihood ratio ${\dx[P_{\theta_1}]}/{\dx[P_{\theta_2}]}$ is nondecreasing on the ordered union of the supports, with the conventions $0/0=0$ and $c/0=+\infty$ for $c>0$.

\end{itemize}
\end{define}

We recall the classical implication between the two orders.
\begin{lemma}[Theorem 1.C.1 in \cite{ShakedShanthikumar}]
\label{lem:lr-implies-st}
If \(P_{\theta_2}\lr P_{\theta_1}\), then \(P_{\theta_2}\st P_{\theta_1}\).
\end{lemma}
By \eqref{def.density_ratio}, $\ell_{\theta_1, \theta_2}$ is proportional to the likelihood ratio with positive constant, so their monotonicity coincides. Therefore, the stochastic orders can be inferred from the weight ratio alone.\footnote{That's the reason we borrowed the standard notation of likelihood ratio to represent the weight ratio.}
\begin{lemma}
\label{lem:weight-order}
Let \(X_i\sim P_{\theta_i}\) for $i=1,2$.  If the weight ratio $\ell_{\theta_1, \theta_2}$ is nondecreasing, then \(P_{\theta_2}\st P_{\theta_1}\), so \(\E\,h(X_2)\le\E\,h(X_1)\) for every nondecreasing \(h\) for which both expectations exist.
\end{lemma}
\begin{proof}
Since $\ell_{\theta_1, \theta_2}$ is nondecreasing, the likelihood ratio is nondecreasing, so \(P_{\theta_2}\lr P_{\theta_1}\) by Definition~\ref{def:orders}. Hence, \(P_{\theta_2}\st P_{\theta_1}\) by Lemma~\ref{lem:lr-implies-st}.
\end{proof}


In order to justify the validity of \eqref{def.kernel} and, more generally, differentiating \(\Phi\) under the sum or integral, we postulate the following regularity conditions.

\begin{assumption}
\label{asm:regularity}
For each parameter \(r \in \{\lambda, \theta\}\) with respect to which we differentiate, the map \(r\mapsto\log w_{\lambda,\theta}(t)\), with the other parameters fixed, is \(C^1\) for every \(t\in J\), and for every compact interval \(B\) in the domain of \(r\) there exists a majorant \(m_B\in L^1(J,\mu)\) such that
\[
        \sup_{r\in B}\bigl|\partial_r w_{\lambda,\theta}(t)\bigr|
        \le m_B(t)
        \qquad\text{for \(\mu\)-almost every }t\in J .
\]
\end{assumption}
By dominated convergence this permits differentiation under the sign in any single parameter, which we use without further mention. For the mixed-parameter log-derivative treated in the subsequent sections, we impose the following additional condition to justify the mixed differentiation and the interchange of \(\partial_\lambda\) and \(\partial_\theta\).
\begin{assumption}
\label{asm:mixed-regularity}
The map \((\lambda,\theta)\mapsto\log w_{\lambda,\theta}(t)\) is \(C^2\) on \(\Lambda\times\Theta\) for every \(t\in J\), and for every compact rectangle \(B\subseteq\Lambda\times\Theta\) there exists a majorant \(m_B\in L^1(J,\mu)\) such that
\[
        \sup_{(\lambda,\theta)\in B}
        \bigl|\partial_\lambda\partial_\theta w_{\lambda,\theta}(t)\bigr|
        \le m_B(t)
        \qquad\text{for \(\mu\)-almost every }t\in J .
\]
\end{assumption}


\vspace{5pt}

We conclude the preliminaries with sign results used repeatedly below. In what follows, $I \subseteq \mr$.

\begin{lemma}
\label{lem.sign.logderiv}
If \(f:I\to(0,\infty)\) is differentiable on $I$, then for every \(t\in I\)
\[
        \sgn\, f'(t) = \sgn\,(\log {f(t)})'.
\]
\end{lemma}
\begin{proof}
Since $f>0$, the assertion is implied by the identity $(\log f(t))'={f'(t)}/{f(t)}$.
\end{proof}
\begin{lemma}
\label{lem.sumproduct}
Given two functions \(f: I \to \mr\) and \(g: I \to(0,\infty)\), for every \(t_1,t_2\in I\) it holds that
\[
        \sgn\{f(t_2)g(t_1)-f(t_1)g(t_2)\} = \sgn\left\{ \frac{f(t_2)}{g(t_2)} - \frac{f(t_1)}{g(t_1)} \right\}.
\]
If \(f/g\) is nondecreasing, then for $t_2\ge t_1$
\[
        f(t_2)g(t_1)-f(t_1)g(t_2)\ge0,
\]
and the inequality is reversed if \(f/g\) is nonincreasing.
\end{lemma}
\begin{proof}
Since  \(g>0\) and $f(t_2)g(t_1)-f(t_1)g(t_2) = g(t_1)g(t_2) \left\{ {f(t_2)}/{g(t_2)} - {f(t_1)}/{g(t_1)} \right\},$ the factor \(g(t_1)g(t_2)\) is positive, so the signs agree. Consequently, for $t_2 \ge t_1$ and $f/g$ nondecreasing (resp. nonincreasing), the left-hand side is nonnegative (resp. nonpositive).
\end{proof}

The following is the classical Chebyshev integral inequality.
\begin{lemma}
\label{lem.chebyshev}
Let  $f,g$ be two real-valued functions on $I$ and let $X \sim P$, where $P$ is a probability law on $I$. Assume that $f,g, fg \in L^1(P)$. If $f$ and $g$ are monotone then
\[
        \varepsilon_f\varepsilon_g\, \Cov(f(X),g(X))\ge 0 .
\]
\end{lemma}
\begin{proof}
Let \(X'\) be an independent copy of $X$ and set $Y \coloneqq (f(X)-f(X'))(g(X)-g(X'))$. If \(f\) and \(g\) are monotone in the same direction then $Y \ge 0$ almost surely. If instead they are monotone in opposite directions, then $Y \le 0$ almost surely. Therefore, $\varepsilon_f\varepsilon_g Y\ge0$ almost surely. Taking expectations yields \(\varepsilon_f\varepsilon_g\,\E Y\ge0\), which proves the assertions since
\[
       \E Y  = 2\left\{\E[f(X)g(X)]-\E[f(X)]\,\E[g(X)]\right\} = 2\,\Cov(f(X),g(X)),
\]
\end{proof}

\subsection{The one-kernel Sign Rule}
\label{sub.onekernel}

In this section, Assumption~\ref{asm:regularity} is in force.  
\begin{assertion}
\label{lem:score-identity-final}
Fix $\theta \in \Theta$ and suppose \(X \sim P_{\theta}\), where $P_{\theta}$ is the law  defined in \eqref{def.law}. Then
\[
\partial_\lambda\log
\Phi(\lambda,\theta) = \E\, K_{\lambda}(X).
\]
\end{assertion}
\begin{proof}
Since differentiating under the integral is justified by Assumption~\ref{asm:regularity}, we have
\[
\partial_\lambda \Phi(\lambda,\theta)
=
\int_J
\partial_\lambda w_{\lambda,\theta}(t)\,\mu(\dx[t]).
\]
Dividing both sides by \(\Phi(\lambda,\theta)\) and using \eqref{def.kernel}, we obtain
\begin{alignat*}{2}
\partial_\lambda\log \Phi(\lambda,\theta)&=
 \int_J K_\lambda(t)  P_\theta(\dx[t]) =\E\, K_{\lambda}(X).
\end{alignat*}
\end{proof}
\begin{cor}
Assume $K_\lambda$ independent of \(\theta\). For every \(\theta_1,\theta_2\in\Theta\),
\label{thm:expectation-final}
\[
\partial_\lambda\log\frac{\Phi(\lambda,\theta_1)}{\Phi(\lambda,\theta_2)}
= \E\, K_{\lambda}(X_1)-\E\, K_{\lambda}(X_2),
\]
where \(X_i\sim P_{\theta_i}\)  for $i=1,2$.%
\end{cor}
\begin{proof}
Take the difference of the two identities of Proposition~\ref{lem:score-identity-final} at \(\theta_1\) and \(\theta_2\).
\end{proof}
By Lemma~\ref{lem:weight-order}, monotonicity of the weight ratio orders the two laws. The sign of the quotient's derivative is determined by the expectation difference of a monotone kernel under the two laws.
\begin{thm}
\label{thm:lr-final}
Fix $\theta_1, \theta_2 \in \Theta$, suppose \(K_\lambda\) is independent of \(\theta\), and it is nondecreasing (resp. nonincreasing) in \(t\).  If $\Ptwo\st \Pone$  for every \(\lambda\in\Lambda\), the quotient
\[
        \lambda\longmapsto
        \frac{\Phi(\lambda,\theta_1)}{\Phi(\lambda,\theta_2)}
\]
is nondecreasing (resp. nonincreasing) on \(\Lambda\).%
\end{thm}

\begin{proof}
Let \(X_i\sim P_{\theta_i}\) for $i=1,2$.  By Corollary~\ref{thm:expectation-final},
\[
        \partial_\lambda\log
        \frac{\Phi(\lambda,\theta_1)}{\Phi(\lambda,\theta_2)}
        =
        \E\, K_\lambda(X_1)-\E\, K_\lambda(X_2).
\]
If $K_\lambda$ is nondecreasing, then the usual stochastic order compares expectations of nondecreasing functions, so \(\Ptwo\st \Pone\) yields \(\E\,K_\lambda(X_2)\le\E\,K_\lambda(X_1)\).  This is a nonnegative derivative of the logarithm of the quotient, hence the quotient is nondecreasing by Lemma~\ref{lem.sign.logderiv}.  The nonincreasing kernel case reverses the sign, which concludes the proof.
\end{proof}

Another useful result is a differentiation formula for the expectation of a test function under the induced law. When the test function itself depends on $\lambda$, the derivative then carries an extra covariance with the kernel, reflecting the dependence of the induced law on $\lambda$
\begin{assertion}
\label{lem:induced-expectation-derivative}
Fix \(\theta\in\Theta\) and let \(X\sim P_\theta\).  Suppose  $h_\lambda:J\to\mathbb R$ be a family of functions differentiable in \(\lambda\), with \(h_\lambda(X)\) and \(\partial_\lambda h_\lambda(X)\) integrable.

Define
$$R(\lambda) \coloneqq \E\,h_\lambda(X).$$
If \(h_\lambda w_{\lambda,\theta}\) satisfies the regularity condition of Assumption~\ref{asm:regularity} in \(\lambda\), then
\[
        R'(\lambda)
        =
        \E\,\partial_\lambda h_\lambda(X)
        +
        \Cov\bigl(h_\lambda(X),K_\lambda(X)\bigr).
\]
\end{assertion}
\begin{proof}
Since \(\partial_\lambda w_{\lambda,\theta}(t)=K_\lambda(t) w_{\lambda,\theta}(t)\) for $t \in J$ by \eqref{def.kernel} and \(\partial_\lambda\Phi^{-1}(\lambda, \theta)=-\Phi^{-1}(\lambda, \theta)\E\,K_\lambda(X)\) by Proposition~\ref{lem:score-identity-final}, proceed by direct differentiation to get
\begin{alignat*}{2}
        R'(\lambda)
        &=\partial_\lambda\int_J h_\lambda(t) \frac{w_{\lambda,\theta}(t)}{\Phi(\lambda, \theta)}\,\dx[\mu]\\
        &=\int_J(\partial_\lambda h_\lambda(t))\frac{w_{\lambda,\theta}(t)}{\Phi(\lambda, \theta)}\,\dx[t]+\int_J h_\lambda(t)  \frac{\partial_\lambda w_{\lambda,\theta}(t)}{\Phi(\lambda, \theta)}\,\dx[t]+\left(\partial_\lambda\Phi^{-1}(\lambda, \theta)\right)\int_J h_\lambda(t) w_{\lambda,\theta}(t)\,\dx[\mu]\\
        &=\E\,\partial_\lambda h_\lambda(X) + \int_J h_\lambda(t) K_\lambda(t) \frac{w_{\lambda,\theta}(t)}{\Phi(\lambda, \theta)}\,\dx[t]-\E\,K_\lambda(X))\int_J h_\lambda(t) \frac{w_{\lambda,\theta}(t)}{\Phi(\lambda, \theta)}\,\dx[\mu]\\
        &=\E\,\partial_\lambda h_\lambda(X)+\mE{h_\lambda(X) K_\lambda(X)} -\E\,K_\lambda(X)\,\E\,h_\lambda(X)\\
        &=\E\,\partial_\lambda h_\lambda(X)+\Cov(h_\lambda(X),K_\lambda(X)).
\end{alignat*}
\end{proof}

\begin{thm}[Sign Rule]
\label{thm:sign-rule}
Assume \(K_\lambda\) is independent of \(\theta\).  If  the kernel \(K_\lambda\) and the weight ratio $\ell_{\theta_1, \theta_2}$ are  monotone  for every \(\lambda\in\Lambda\) then
\[
        \varepsilon_K\,\varepsilon_\ell\,
        \partial_\lambda \log
        \frac{\Phi(\lambda,\theta_1)}{\Phi(\lambda,\theta_2)}
        \ge0 ,
\]
\end{thm}

\begin{proof}
A nondecreasing \(\ell_{\theta_1, \theta_2}\) implies \(\Ptwo\st\Pone\) by Lemma~\ref{lem:weight-order} and the reverse order when it is nonincreasing. Applying Theorem~\ref{thm:lr-final} with \(K_\lambda\) nondecreasing and then with \(K_\lambda\) nonincreasing yields four combinations encoded by the stated inequality.
\end{proof}

\begin{rem}
\label{rem:weaker-than-lr}
Monotonicity of $\ell_{\theta_1, \theta_2}$  is sufficient but not necessary. Weaker shape conditions on the likelihood ratio, transferable to $\ell_{\theta_1, \theta_2}$, are discussed in~\cite{DerbaziPairwise, Derbazi2026kernel}.
\end{rem}

\begin{cor}[Biernacki--Krzyż form]
\label{cor:gen-bk}
Let \(v_1, v_2: \Lambda \times J \to(0,\infty)\) be differentiable in the first argument such that for each $\lambda \in \Lambda$, $\partial_\lambda\log v_1(\lambda,t) = \partial_\lambda\log v_2(\lambda,t) = K_\lambda(t)$, for all $t\in J.$

Define
\[
F_i(\lambda)=\int_J v_i(\lambda,t)\,\dx[\mu(t)],
\qquad i=1,2,
\]
and
\[
        \ell(t) \coloneqq \frac{v_1(\lambda,t)}{v_2(\lambda,t)}.
\]
If \(K_\lambda\) and \(\ell\) are both monotone for every \(\lambda\in\Lambda\), then
\[
        \varepsilon_K\varepsilon_\ell\,
        \partial_\lambda\frac{F_1(\lambda)}{F_2(\lambda)}
        \ge0 .
\]
\end{cor}

\begin{proof}
Take $w_{\lambda, i}(t) = v_i(\lambda, t)$. The result follows by application of Theorem~\ref{thm:sign-rule} since the common kernel is independent of the shifted parameter \(\theta\in\{1,2\}\).
\end{proof}

\subsection{The two-kernel Sign Rule}
\label{subsec.two_kernel}
In this section, Assumptions~\ref{asm:regularity} and~\ref{asm:mixed-regularity} are in force. Differentiating \(\log\Phi\) twice, once in $\lambda$ and once in \(\theta\), gives the following covariance relation.

\begin{assertion}
\label{lem:mixed-identity}
Fix \((\lambda,\theta) \in \Lambda \times \Theta\) and let \(X\sim P_\theta\). If $K_\lambda$ is independent of $\theta$, then
\[
        \partial_\theta\partial_\lambda\log\Phi(\lambda,\theta)
        =\Cov(K_\lambda(X),L_\theta(X)).
\]
\end{assertion}
\begin{proof}
By Proposition~\ref{lem:score-identity-final},
$
        \partial_\lambda\log\Phi = \E\,K_\lambda(X) = \int_J K_\lambda(t)\,P_\theta(\dx[t]).
$
Write \(p_\theta(t)=w_{\lambda,\theta}(t)/\Phi(\lambda,\theta)\) for the density of \(P_\theta\) with respect to the reference measure $\mu$.  Differentiating this density in \(\theta\) yields
\[
        \partial_\theta p_\theta(t)
        =
        p_\theta(t)\bigl(L_\theta(t)-\E\,L_\theta(X)\bigr).
\]
Since the differentiation in \(\theta\) and the interchange of \(\partial_\lambda\) and \(\partial_\theta\) is justified by Assumptions~\ref{asm:regularity} and~\ref{asm:mixed-regularity}, it follows that
\[
        \partial_\theta\partial_\lambda\log\Phi
        =\int_J K_\lambda(t)\,\partial_\theta p_\theta(t)\,\dx[\mu(t)]
        =\E\,\bigl[K_\lambda\,(L_\theta-\E\,L_\theta)\bigr]
        =\E\,[K_\lambda L_\theta]-\E\,K_\lambda\,\E\,L_\theta ,
\]
\end{proof}
\begin{thm}[Two-kernel Sign Rule]
\label{thm:mixed}
Assume that \(K_\lambda\) is independent of \(\theta\).  If the kernels \(K_\lambda\) and \(L_\theta\) are monotone, with signs \(\varepsilon_K\) and \(\varepsilon_L\), respectively, then
\[
        \varepsilon_K\,\varepsilon_L\,
        \partial_\theta\partial_\lambda\log\Phi(\lambda,\theta)
        \ge0 .
\]
\end{thm}
\begin{proof}
Apply Chebyshev's integral inequality (Lemma~\ref{lem.chebyshev}) to the law \(P_\theta\) with \(f=K_\lambda\), \(g=L_\theta\), and use the identity of Proposition~\ref{lem:mixed-identity}.
\end{proof}

When the kernels are strictly monotone, the covariance identity also gives the equality case.

\begin{cor}
\label{cor:mixed-equality}
Under the hypotheses of Theorem~\ref{thm:mixed},
whenever \(K_\lambda\) and \(L_\theta\) are strictly monotone and the support of \(P_\theta\) has at least two points, the Sign Rule inequality is strict.
\end{cor}
\begin{proof}
By Proposition~\ref{lem:mixed-identity}, \(\partial_\theta\partial_\lambda\log\Phi=\Cov(K_\lambda(X),L_\theta(X))\). Equality in the Chebyshev integral inequality for monotone \(K_\lambda,L_\theta\) forces \(K_\lambda(X)\) or \(L_\theta(X)\) to be constant almost surely. Since \(K_\lambda\) and \(L_\theta\) are strictly monotone, either image is constant only if \(X\) is concentrated on a single point of \(J\), so equality forces \(X\) to be \(P_\theta\)-almost surely constant. Conversely, if the support of \(P_\theta\) has at least two points, then \(X\) is nondegenerate, both images are nondegenerate strictly monotone functions of \(X\), and the covariance has the strict sign \(\varepsilon_K\varepsilon_L\).
\end{proof}

The sign of the mixed derivative is the product of the two kernel signs, the same form the Sign Rule uses, with the kernel \(L_\theta\) in the shifted parameter in place of the weight ratio.  The next corollary states the determinantal form.
\begin{cor}
\label{cor:determinant}
Under the hypotheses of Theorem~\ref{thm:mixed}, for \(\xi_1\le\xi_2\) in \(\Lambda\) and \(\eta_1\le\eta_2\) in \(\Theta\),
\[
        \varepsilon_K\,\varepsilon_L\,
        \biggl\{
        \Phi(\xi_2,\eta_2)\Phi(\xi_1,\eta_1)
        -
        \Phi(\xi_1,\eta_2)\Phi(\xi_2,\eta_1)\biggr\}\ge0 .
\]
Thus, \(\Phi\) is log-supermodular in \((\lambda,\theta)\) when \(\varepsilon_K\varepsilon_L=+1\), and log-submodular when \(\varepsilon_K\varepsilon_L=-1\). The inequality is strict on nondegenerate rectangles whenever the strict covariance sign of Corollary~\ref{cor:mixed-equality} holds throughout the rectangle.
\end{cor}

\begin{proof}
Theorem~\ref{thm:mixed} gives
$
        \varepsilon_K\varepsilon_L\,
        \partial_\theta\partial_\lambda\log\Phi(\lambda,\theta)\ge0,
$
and since \(\log\Phi\) is \(C^2\) on $\Lambda \times \Theta$ by Assumption~\ref{asm:mixed-regularity}, integrate over the rectangle \([\xi_1,\xi_2]\times[\eta_1,\eta_2]\) to obtain
\[
\varepsilon_K\varepsilon_L
\log\frac{\Phi(\xi_2,\eta_2)\Phi(\xi_1,\eta_1)}
        {\Phi(\xi_1,\eta_2)\Phi(\xi_2,\eta_1)}
\ge0,
\]
which is equivalent to
\[
\varepsilon_K\varepsilon_L
\left\{
\frac{\Phi(\xi_2,\eta_2)\Phi(\xi_1,\eta_1)}
     {\Phi(\xi_1,\eta_2)\Phi(\xi_2,\eta_1)}
-1
\right\}\ge0.
\]
Multiply both sides by \(\Phi(\xi_1,\eta_2)\Phi(\xi_2,\eta_1)\), which is positive, to get the stated inequality.
\end{proof}
\begin{rem}
\label{rem.shift}
Suppose \(\Phi(\lambda,\theta)=F(\lambda+\theta)\). For \(\xi_1\le\xi_2\), \(\eta_1\le\eta_2\) set $u=\xi_1+\eta_1$, $\alpha=\xi_2-\xi_1,$ and $\beta=\eta_2-\eta_1$. In this setting, the determinant in Corollary~\ref{cor:determinant} becomes
\[
        F(u+\alpha+\beta)F(u)-F(u+\alpha)F(u+\beta),
        \qquad \alpha,\beta\ge0 .
\]
Thus, in the one-parameter case, the determinant inequality is the Tur\'an-type form of log-convexity of \(F\), with the opposite sign for log-concavity.
\end{rem}

A useful special case is that of positive power series. Let
\[
        \Phi(x,\theta)=\sum_{n\ge0} c_n(\theta)x^n,
        \qquad c_n(\theta)>0,
\]
for \(0<x<R(\theta)\), where \(R(\theta)\) denotes the radius of convergence, so that \(\Phi(x,\theta)<\infty\). The argument $x$ enters every term as \(x^n\), so its kernel is the index itself,
\[
        K_x(n)=\partial_x\log\bigl(c_n(\theta)x^n\bigr)=\frac{n}{x},
\]
increasing in \(n\) regardless of the family.  The mixed derivative against any parameter \(\theta\) is therefore governed by that parameter's kernel alone.

\begin{cor}
\label{cor:argument-mixed}
For a positive power series in \(x>0\), with \(\theta\)-kernel \(L_\theta\) of sign \(\varepsilon_L\),
\[
        \varepsilon_L\,\partial_x\partial_\theta\log\Phi(x, \theta)\ge0 .
\]
In particular \(\Phi\) is log-supermodular in \((x,\theta)\) when \(L_\theta\) is increasing, and log-submodular when it is decreasing.
\end{cor}
\begin{proof}
Apply Theorem~\ref{thm:mixed} with \(\lambda=x\). Since \(K_x(n)=n/x\) is increasing,  \(\varepsilon_K=+1\) and hence, \(\varepsilon_K\varepsilon_L=\varepsilon_L\).
\end{proof}

\subsection{The variance rule and log-convexity}
\label{subsec.variance-rule}
Consider now the one-parameter function
\[\Phi(\lambda)=\int_J w_\lambda(t)\,\mu(\dx[t]),\]
where $w_\lambda: J \to (0, \infty)$. As before, $$K_\lambda(t)=\partial_\lambda\log w_\lambda(t).$$
In this setting, the curvature of $\Phi$ is governed by both the variance and the derivative of $K_\lambda$.
\begin{lemma}
        \label{lem:variance}
Let \(X\sim P_\lambda\), where $P_\lambda$ is the probability law whose density, with respect to $\mu$, is \(w_{\lambda}(t)/\Phi(\lambda)\). If the same regularity condition holds for \(K_\lambda w_\lambda\) in \(\lambda\), then
\[
 \partial_\lambda^2\log\Phi(\lambda)
 =
 \operatorname{Var}\bigl(K_\lambda(X)\bigr)+\E\,\partial_\lambda K_\lambda(X).
\]
If in addition \(\partial_\lambda K_\lambda(t)\ge0\) for all \(t\), then \(\Phi\) is log-convex in \(\lambda\). In particular, when \(K_\lambda(t)\) is independent of \(\lambda\),
\[
 \partial_\lambda^2\log\Phi(\lambda)  =\operatorname{Var}\bigl(K_\lambda(X)\bigr),
\]
and \(\Phi\) is log-convex in \(\lambda\).
\end{lemma}
\begin{proof}
By Proposition~\ref{lem:score-identity-final}, \(\partial_\lambda\log\Phi(\lambda)=\E\,K_\lambda(X)\). Applying Proposition~\ref{lem:induced-expectation-derivative} with \(h_\lambda(t)=K_\lambda(t) \) for $t \in J$ produces
\[
\begin{aligned}
        \partial_\lambda^2\log\Phi(\lambda)
        &=
        \E\,\partial_\lambda K_\lambda(X)
        +
        \Cov(K_\lambda(X),K_\lambda(X))  \\
        &=
        \E\,\partial_\lambda K_\lambda(X)
        +
        \Var(K_\lambda(X)),
\end{aligned}
\]
which is the required identity. For the second assertion, if \(\partial_\lambda K_\lambda(t)\ge0\) pointwise, then the variance term and the expectation term are both nonnegative, so \(\partial_\lambda^2\log\Phi(\lambda)\ge0\). Finally, if \(K_\lambda\) is independent of \(\lambda\), the expectation term vanishes and the variance identity follows.
\end{proof}
The last assertion of the corollary highlights a common special case. A basic example is the Laplace-transform weight \(e^{-\lambda t}\). Any positive mixture of these weights is log-convex.
\begin{lemma}
\label{lem:positive-transform}
Let \(I\subseteq\mathbb R\) be an open interval and let \(\nu\) be a nonnegative measure on \(\mathbb R\). If
\[
        F(\lambda)=\int_{\mathbb R} e^{-\lambda u}\,\nu(\dx[u])
\]
is finite and positive for every \(\lambda\in I\), then \(F\) is log-convex on \(I\).
\end{lemma}
\begin{proof}
Fix \(\lambda\in I\). The induced law is $$P_\lambda(\dx[u])=F(\lambda)^{-1}e^{-\lambda u}\nu(\dx[u]), \quad u \in I.$$

Since \(F\) is finite on the open interval \(I\), for each \(\lambda\in I\) the induced law has moment generating function
\[
        M_X(s)= \frac{1}{F(\lambda)}
 \int_{\mathbb R} e^{s u}e^{-\lambda u}\nu(\dx[u])=\frac{F(\lambda-s)}{F(\lambda)}
\]
finite for all sufficiently small \(s\). Hence, $X$ has finite second moment.

Now, the associated kernel is \(K_\lambda(t)=\partial_\lambda\log e^{-\lambda t}=-t\), independent of \(\lambda\). Lemma~\ref{lem:variance} then yields \(\partial_\lambda^2\log F(\lambda)=\Var(-X)=\Var(X)\ge0\) for $X \sim P_\lambda$. Thus,  \(F\) is log-convex.
\end{proof}
\begin{example}[Bohr--Mollerup theorem - log-convexity of the gamma function]
\label{ex:gamma-logconvex}
For \(\lambda>0\), the Eulerian integral \(\Gamma(\lambda)=\int_0^\infty t^{\lambda-1}e^{-t}\,\dx[t]\) has weight function \(w_\lambda(t)=t^{\lambda-1}e^{-t}\), whose \(\lambda\)-kernel is \(K_\lambda(t)=\partial_\lambda\log w_\lambda(t)=\log t\), independent of \(\lambda\).  With \(X\) distributed according to the gamma law \(w_\lambda(t)/\Gamma(\lambda)\), Lemma~\ref{lem:variance} gives \((\log\Gamma)''(\lambda)=\Var(\log X)=\psi'(\lambda)\ge0\) and the ensuing log-convexity of the Gamma function.
\end{example}

\subsection{Truncated series and integrals}
\label{sec:truncation-final}
Let \(w_\lambda:J\to(0,\infty)\) and define the lower and upper truncations
\begin{equation}
\label{def.truncated_laws}
        \Phi^-(\lambda,x)
        =
        \int_{J\cap(-\infty,x]} w_\lambda(t)\,\mu(\dx[t]),
        \qquad
        \Phi^+(\lambda,x)
        =
        \int_{J\cap[x,\infty)} w_\lambda(t)\,\mu(\dx[t]),
\end{equation}
assuming both quantities are finite and positive.  The truncation point \(x\) is now the shifted parameter, and the weight function itself does not depend on \(x\) and so is the corresponding \(\lambda\)-kernel.

Define the associated probability laws \(P^-_{x}\), \(P^+_{x}\) by
\begin{equation}
P^-_{x}(\dx[t]) =
\frac{\1{\{t\le x\}}w_{\lambda}(t)}{\Phi^-(\lambda,x)}\,\dx[\mu(t)],
\qquad
P^+_{x}(\dx[t])  =
\frac{\1{\{t\ge x\}}w_{\lambda}(t)}{\Phi^+(\lambda,x)}\,\dx[\mu(t)],
\end{equation}
where $\1\{\cdots\}$ denotes the identity function. In this setting, \(\lambda\) is fixed, so the ordering below is in the truncation parameter \(x\).
\begin{lemma}
\label{lem:trunc-lr-final}
For \(x_1<x_2\),
\[
        P_{x_1}^-\st P_{x_2}^-,
        \, \text{ and }\,
        P_{x_1}^+\st P_{x_2}^+ .
\]
\end{lemma}
\begin{proof}
For the lower truncation, the likelihood ratio \(\dx[P^-_{x_2}]/\dx[P^-_{x_1}]\), extended to the ordered union of supports, is constant on \(J\cap(-\infty,x_1]\) and \(+\infty\) on \(J\cap(x_1,x_2]\).  Hence,  \(P^-_{x_1}\lr P^-_{x_2}\) which implies \(P^-_{x_1}\st P^-_{x_2}\) by Lemma~\ref{lem:lr-implies-st}.  The second assertion follows by a similar argument.
\end{proof}

\begin{cor}
\label{cor:truncation-sign-rule}
If \(K_\lambda\) is monotone with sign \(\varepsilon_K\), then for \(x_1<x_2\) the quotients
\[
        \lambda\longmapsto
        \frac{\Phi^-(\lambda,x_2)}{\Phi^-(\lambda,x_1)},
        \qquad
        \lambda\longmapsto
        \frac{\Phi^+(\lambda,x_2)}{\Phi^+(\lambda,x_1)}
\]
are monotone in \(\lambda\) with sign \(\varepsilon_K\).  Consequently, for \(\xi_1\le\xi_2\), each truncated function \(\Phi^\pm\) satisfies
\begin{equation}
\label{eq.trunc-determinant}
        \varepsilon_K\,
        \bigl\{
        \Phi^\pm(\xi_2,x_2)\Phi^\pm(\xi_1,x_1)
        -\Phi^\pm(\xi_1,x_2)\Phi^\pm(\xi_2,x_1)\bigr\}\ge0 ,
\end{equation}
so \(\Phi^-\) and \(\Phi^+\) are log-supermodular in \((\lambda,x)\) when \(K_\lambda\) is nondecreasing, and log-submodular when \(K_\lambda\) is nonincreasing.
\end{cor}
\begin{proof}
By Lemma~\ref{lem:trunc-lr-final}, \(P_{x_1}^\pm\st P_{x_2}^\pm\).  For the first assertion, apply Theorem~\ref{thm:lr-final} with \(\theta_1=x_2\) and \(\theta_2=x_1\) to get the stated quotient monotonicity.  For the second part, apply Lemma~\ref{lem.sumproduct} to \(\lambda\longmapsto\Phi^\pm(\lambda,x_2)/\Phi^\pm(\lambda,x_1)\)  to obtain \eqref{eq.trunc-determinant}.
\end{proof}

\begin{example}[The incomplete gamma function]
\label{ex:incomplete-gamma}
The lower incomplete gamma function \(\gamma(\lambda,x)=\int_0^x t^{\lambda-1}e^{-t}\,\dx[t]\) is the lower truncation at \(x\) of the gamma weight \(w_\lambda(t)=t^{\lambda-1}e^{-t}\), with \(\lambda\)-kernel \(K_\lambda(t)=\log t\), increasing, so \(\varepsilon_K=1\). By Corollary~\ref{cor:truncation-sign-rule}, for \(0<x_1<x_2\) the quotient \(\lambda\mapsto\gamma(\lambda,x_2)/\gamma(\lambda,x_1)\) is increasing on \((0,\infty)\), and \(\gamma\) is log-supermodular in \((\lambda,x)\).
\end{example}

When the monotonicity variable is the discrete truncation index, there is no $C^1$ log-derivative in that direction, so the kernel $K_\lambda$ and its additive mean identity $\E\,K_\lambda(t)=\partial_\lambda\log\Phi(\lambda, \theta)$ are unavailable. In this case, we substitute them by a multiplicative counterpart as follows: take \(J=\mathbb N_0\), or more generally an integer tail, and write
\[
        \Phi^+(\lambda,m)=\sum_{n\ge m}w_\lambda(n).
\]
Let \(P^+_m\) be the corresponding upper-truncation law and let $X\sim P^+_m$.  The forward ratio
\begin{equation*}
{\cal K}(n)=\frac{w_\lambda(n+1)}{w_\lambda(n)}
\end{equation*}
serves as a discrete kernel, with the `multiplicative' mean identity
\begin{equation}
        \label{def.expectation.identity}
        \E\,{\cal K}(X)=\frac{\Phi^+(\lambda,m+1)}{\Phi^+(\lambda,m)},
\end{equation}
in which the forward ratio of \(\Phi^+\) is multiplicative analogue of \(\partial_\lambda\log\Phi\).
\begin{lemma}[Discrete Sign Rule for tail sums]
\label{lem:discrete-tail-sign-rule}
If the discrete kernel \({\cal K}\) is monotone,  with sign \(\varepsilon_{\cal K}\), then
\[
        -\varepsilon_{\cal K}\,
        \bigl(\Phi^+(\lambda,m+1)^2-\Phi^+(\lambda,m)\,\Phi^+(\lambda,m+2)\bigr)
        \ge0,
        \qquad m\in\mathbb N_0 .
\]
\end{lemma}
\begin{proof}
By Lemma~\ref{lem:trunc-lr-final} the truncation laws satisfy \(P^+_m\st P^+_{m+1}\). Let $X\sim P^+_m$ and $X'\sim P^+_{m+1}$.  If \({\cal K}\) is nonincreasing, comparing its means under the two ordered laws, through the mean identity \eqref{def.expectation.identity} gives
\[
        \frac{\Phi^+(\lambda,m+2)}{\Phi^+(\lambda,m+1)}
        =
        \E\,{\cal K}(X')
        \le
        \E\,{\cal K}(X)
        =
        \frac{\Phi^+(\lambda,m+1)}{\Phi^+(\lambda,m)},
\]
which is \(\Phi^+(\lambda,m+1)^2\ge\Phi^+(\lambda,m)\Phi^+(\lambda,m+2)\), the case \(\varepsilon_{\cal K}=-1\).  The nondecreasing case is identical with the inequality reversed.
\end{proof}

\begin{rem}
The log-concave case is the standard \(PF_2\) closure property for upper tails (equivalently, for log-concave sequences with no internal zeros).  The proof above establishes the same fact in the present language: the tail ratio is the expectation of the ratio-kernel \({\cal K}\), and the upper truncation laws are stochastically ordered.  The log-convex case follows from the same argument with the monotonicity reversed.
\end{rem}

\begin{rem}
\label{rem.covariance.upper}
The inequality of Lemma~\ref{lem:discrete-tail-sign-rule} is also a single-law covariance. Let $X\sim P^+_m$. Keeping the upper-truncation law \(P^+_m\) and the discrete kernel \({\cal K}(n)=w_\lambda(n+1)/w_\lambda(n)\), the increment of the tail ratio is a covariance under \(P^+_m\): since deleting the bottom atom gives \(P^+_{m+1}\) as \(P^+_m\) conditioned on \(\{X>m\}\),
\[
        \Phi^+(\lambda,m+1)^2-\Phi^+(\lambda,m)\,\Phi^+(\lambda,m+2)
        =
        -\,\frac{\Phi^+(\lambda,m)\,\Phi^+(\lambda,m+1)}{P^+_m(X>m)}\,
        \Cov\!\bigl({\cal K}(X),\,\1{\{X>m\}}\bigr).
\]
When \({\cal K}\) is nonincreasing the indicator \(\1{\{X>m\}}\) is nondecreasing, so the covariance is nonpositive by the Chebyshev integral inequality (Lemma~\ref{lem.chebyshev}).  The determinant is thus a Chebyshev covariance of the kernel with the truncation indicator, an alternative to the upper-tail Sign Rule that invokes no stochastic order.
\end{rem}

\section{Applications}
\label{sec:sample-classical}

\subsection{Quotients of generalised hypergeometric functions}
\label{ex.gen-hypergeometric}

Let
\[
\HypF pq\!\left(
\genfrac{}{}{0pt}{} {a_1,\ldots,a_p}{b_1,\ldots,b_q}
\middle| x\right)
\coloneqq
\sum_{n=0}^\infty
\frac{\prod_{i=1}^p(a_i)_n}
     {\prod_{j=1}^q(b_j)_n}
\frac{x^n}{n!},
\]
with all the parameters positive and \(x>0\) in the convergence domain.

Single out two of the parameters, a monotonicity parameter \(\lambda\) and a distinct shifted parameter \(\theta\), each located in a fixed upper or lower position, with all the remaining parameters held fixed.  Let \(Q_{\lambda;\theta_1,\theta_2}(x)\) be the quotient of two such functions that agree everywhere except in the \(\theta\) position, which holds the value \(\theta_1\) in the numerator and \(\theta_2\) in the denominator.  Normalising the terms gives, in either function, the induced law on \(\Nzero\)
\[
        P(n)\propto w(n)\coloneqq\frac{\prod_{i=1}^p(a_i)_n}{\prod_{j=1}^q(b_j)_n}\,\frac{x^n}{n!},
\]
and the quotient is compared through the ratio of these two laws.  The kernel of \(\lambda\) is \(n\mapsto\psi(\lambda+n)-\psi(\lambda)\), increasing in \(n\), when \(\lambda\) is upper, and its negative, decreasing, when \(\lambda\) is lower. Since the two functions differ only in the \(\theta\) position, the ratio of their \(n\)th terms reduces to \(n\mapsto(\theta_1)_n/(\theta_2)_n\) when \(\theta\) is upper and \(n\mapsto(\theta_2)_n/(\theta_1)_n\) when \(\theta\) is lower, increasing in \(n\) exactly when \(\theta_1>\theta_2\) and \(\theta_2>\theta_1\) respectively.  The term \(w(n)\) is smooth in each parameter, and on every compact parameter set, the relevant first and mixed derivatives are dominated by a summable majorant inside the disc of convergence.  Hence the regularity assumptions hold, and the Sign Rule (Theorem~\ref{thm:sign-rule}) delivers
\[
        \sgn(\theta_1-\theta_2)\,
        \partial_\lambda Q_{\lambda;\theta_1,\theta_2}(x)\ge0
        \quad\text{when \(\lambda,\theta\) are of the same type,}
\]
with the inequality reversed when they are of opposite type. The same upper/lower classification applies to the \(\theta\)-kernel.  By the two-kernel Sign Rule (Theorem~\ref{thm:mixed}) the mixed derivative is the covariance of the two parameter kernels,
\[
        \varepsilon_K\varepsilon_L\,\partial_\lambda\partial_\theta\log\HypF pq\ge0,
\]
nonnegative when \(\lambda,\theta\) are of the same type and nonpositive when of opposite type.  Its `finite-difference' form is the determinant of Corollary~\ref{cor:determinant}: for increments \(\alpha,\beta\ge0\), in the upper--upper Gaussian case, for \(0<x<1\), \(a,b,c>0\),
\[
        \HypF21(a+\alpha,b+\beta;c;x)\HypF21(a,b;c;x)
        -
        \HypF21(a,b+\beta;c;x)\HypF21(a+\alpha,b;c;x)
        \ge0,
\]
and, in the upper--lower Kummer case, for \(x>0\), \(a,c>0\),
\[
        \HypF11(a+\alpha;c+\beta;x)\HypF11(a;c;x)
        -
        \HypF11(a;c+\beta;x)\HypF11(a+\alpha;c;x)
        \le0 ,
\]
the sign reversed from the upper--upper case because \(\lambda\) and \(\theta\) are of opposite type.

\noindent When the monotonicity variable and the shifted parameter coincide, the variance
rule applies to lower parameters but not to upper ones.  For a lower parameter \(b_j=\lambda\), the kernel \(K_\lambda(n)=-\{\psi(\lambda+n)-\psi(\lambda)\}\) has parameter derivative \(\partial_\lambda K_\lambda(n)=\psi'(\lambda)-\psi'(\lambda+n)\ge0\), so Lemma~\ref{lem:variance} gives \(\partial_\lambda^2\log\HypF pq\ge0\). Hence, the function is log-convex in each lower parameter, and the quotient \(\HypF pq(\dots;\lambda+\delta;\dots)/\HypF pq(\dots;\lambda;\dots)\) is increasing in \(\lambda\). This result covers precisely the third quotient from the introductory example.

For an upper parameter, the variance rule is inconclusive, since the variance term and the expectation of \(\partial_\lambda K_\lambda\) have opposite signs.
\subsection{The generalised Stieltjes transform}
\label{ex:karp-sitnik}

In~\cite{KarpSitnik2009ratios}, a shifted generalised hypergeometric function is represented by the generalised Stieltjes transform
\[
        \HypF{q+1}{q}(\sigma,(a_q) + \delta;(b_q)+ \delta;-x)
        =A_\delta\int_0^1 s^{a_1+\delta-1}\,g((a_q)+\delta;(b_q)+\delta;s)\,(1+sx)^{-\sigma}\,\dx[s] ,
\]
valid for \(b_k>a_k>0\), where \(A_\delta>0\) and \(\delta\ge0\).  Thus
\[
        w_{x,\delta}(s)=A_\delta s^{a_1+\delta-1}(1+sx)^{-\sigma}\,g((a_q)+\delta;(b_q)+\delta;s),
        \qquad 0<s<1.
\]
The \(x\)-kernel
\[
        K_x(s)=\partial_x\log(1+sx)^{-\sigma}=-\frac{\sigma s}{1+sx},
\]
is independent of \(\delta\) and, for \(x>-1\) so that \(1+sx>0\) on \((0,1)\), has sign \(\varepsilon_K=-\sgn\sigma\) as a function of \(s\).  Since \(g\) is invariant under simultaneous shifts of all parameters, \(g((a_q)+\delta;(b_q)+\delta;s)=g((a_q);(b_q);s)\) by Lemma~1 of~\cite{KarpSitnik2009ratios}, the weight ratio is \(\ell_{\delta,0}(s)=(A_\delta/A_0)s^\delta\). Hence, \(\varepsilon_\ell=1\) for \(\delta>0\).  The weight \(w_{x,\delta}\) is smooth in \(x\), and dominated convergence applies locally on compact subintervals of \((-1,\infty)\) by the positive Stieltjes representation, so Assumption~\ref{asm:regularity} holds.  For \(x>-1\) and \(\delta>0\), the Sign Rule yields
\[
        -\sgn(\sigma)\,\partial_x
        \frac{\HypF{q+1}{q}(\sigma,(a_q)+\delta;(b_q)+\delta;-x)}
             {\HypF{q+1}{q}(\sigma,(a_q);(b_q);-x)}
        \ge0,
\]
recovering Theorem~1 in~\cite{KarpSitnik2009ratios}.  The \(x\)-kernel has derivative \(\partial_x K_x(s)=\sigma s^2/(1+sx)^2\), nonnegative for \(\sigma>0\), so Lemma~\ref{lem:variance} yields \(\partial_x^2\log\,\HypF{q+1}{q}(\sigma,(a_q);(b_q);-x)\ge0\): the function is log-convex in \(x\) for \(\sigma>0\), \(x>-1\).

\begin{cor}
\label{cor:stieltjes-meijer-relaxation}
Let \(\mathbf a=(a_1,\ldots,a_q)\) and \(\mathbf b=(b_1,\ldots,b_q)\) be two positive sequences, and write \(\Gamma(\mathbf a)=\prod_i\Gamma(a_i)\).  If the Meijer-\(G\) density
\[
        \rho_{\mathbf a,\mathbf b}(s)
        =
        G^{q,0}_{q,q}\!\left(
        s\,\middle|\begin{matrix}\mathbf b\\ \mathbf a\end{matrix}
        \right)
\]
is nonnegative on \((0,1)\), then, for \(\delta>0\), \(x>-1\), and real \(\sigma\),
\[
        -\sgn(\sigma)\,\partial_x
        \frac{\HypF{q+1}{q}(\sigma,\mathbf a+\delta;\mathbf b+\delta;-x)}
             {\HypF{q+1}{q}(\sigma,\mathbf a;\mathbf b;-x)}
        \ge0 .
\]
\end{cor}

\begin{proof}
By the Karp--Prilepkina Stieltjes representation~\cite{Karp2015Stieltjes},
\[
        \HypF{q+1}{q}(\sigma,\mathbf a;\mathbf b;-x)
        =
        \frac{\Gamma(\mathbf b)}{\Gamma(\mathbf a)}
        \int_0^1 (1+xs)^{-\sigma}\rho_{\mathbf a,\mathbf b}(s)\,\frac{\dx[s]}{s}.
\]
The weight with respect to the reference measure \(\dx[s]/s\) is \(\rho_{\mathbf a,\mathbf b}(s)\), with the factor \(\Gamma(\mathbf b)/\Gamma(\mathbf a)\) absorbed into the normalising constant. Using the Meijer-\(G\) shift relation \(\rho_{\mathbf a+\delta,\mathbf b+\delta}(s) = s^\delta\rho_{\mathbf a,\mathbf b}(s)\), the (shifted verus unshifted) weight ratio is
\[
        \ell_{\delta,0}(s)=C_\delta s^\delta,
        \qquad C_\delta>0,
\]
where \(C_\delta\) collects the gamma prefactors. Hence \(\ell_{\delta,0}\) is increasing on \((0,1)\) for \(\delta>0\), so \(\varepsilon_\ell=1\).  The \(x\)-kernel is $K_x(s)=-{\sigma s} / (1+xs),$ with monotonicity direction \(-\sgn\sigma\) on \(0<s<1\) for \(x>-1\).  The Sign Rule delivers the claim, with regularity following from the same local dominated convergence argument as above.
\end{proof}

\begin{rem}
\label{rem:stieltjes-positive-measure}
The proof does not depend on the particular Euler density.  What is needed is a positive family of measures \((\beta_\delta)_{\delta\ge0}\) on \([0,1]\) such that
\[
        \HypF{q+1}{q}(\sigma,\mathbf a+\delta;\mathbf b+\delta;-x)
        =
        \int_0^1(1+xs)^{-\sigma}\,\beta_\delta(\dx[s])
\]
and such that the shifted measure is absolutely continuous with respect to \(\beta_0\), with
\[
        \frac{\dx[\beta_\delta]}{\dx[\beta_0]}(s)=C_\delta s^\delta,
        \qquad C_\delta>0 .
\]
The Euler density and the Meijer-\(G\) density give rise to two positive representing families of this form.  If a chosen Stieltjes density changes sign, that representation does not induce a probability law, and the method applies only if another positive representing family is available.
\end{rem}

The induced law also furnishes two explicit bounds for the function, computed from its single moment \(\mu=\prod_i a_i/b_i\): one extends a bound of Karp and Sitnik, and the other appears to be new.

\begin{assertion}
\label{prop:stieltjes-bounds}
Let \(\mu=\prod_{i=1}^q a_i/b_i\in(0,1)\). If either the Euler density of Section \ref{ex:karp-sitnik} \((b_k>a_k>0)\) or the Meijer-\(G\) density of Corollary~\ref{cor:stieltjes-meijer-relaxation} is nonnegative, then for every \(\sigma>0\) and nonzero \(x>-1\), it holds that
\[
        \bigl(1+x\mu\bigr)^{-\sigma}
        <
        \HypF{q+1}{q}(\sigma,(a_q);(b_q);-x)
        <
        (1-\mu)+\mu\,(1+x)^{-\sigma} .
\]
\end{assertion}
\begin{proof}
The idea is based on bounding a convex function of a random variable using the Jensen and Edmundson--Madansky~\cite{Madansky1959} inequalities. Let \(S\) be a random variable with the probability law induced by the positive Stieltjes density on \((0,1)\), so that $\E S=\mu$ and $$\HypF{q+1}{q}(\sigma,(a_q);(b_q);-x)=\E(1+xS)^{-\sigma}.$$
Observe that \(\phi(s)=(1+xs)^{-\sigma}\) is strictly convex on \( I \coloneqq [0,1]\) since $\phi''(s)=\sigma(\sigma+1)x^2(1+xs)^{-\sigma-2}>0$ for $s \in I $, $\sigma>0$ and  $x\neq0$.  Consequently, Jensen's inequality yields \(\E\,\phi(S)>\phi(\mu)=(1+x\mu)^{-\sigma}\). On the other side, $S$ has bounded support so  the Edmundson--Madansky inequality furnishes
\[
        \E\,\phi(S)<(1-\E S)\,\phi(0)+\E S\,\phi(1)=(1-\mu)+\mu(1+x)^{-\sigma}.
\]
But $S$ is nondegenerate. Therefore, both bounds are strict.
\end{proof}

The lower bound is that of~\cite[Theorem~3]{KarpSitnik2009ratios}, obtained there only for \(\sigma\ge1\). The convexity argument removes that restriction, giving it for every \(\sigma>0\) and on the wider Meijer-\(G\) range of Corollary~\ref{cor:stieltjes-meijer-relaxation}, and so settles Conjecture~1 of~\cite{KarpSitnik2009ratios} on that range.  The upper bound appears to be new.  Neither addresses the broader conjectural condition \(\sum_i(b_i-a_i)>0\).

\subsection{Prabhakar quotients and remainders}
\label{sec:prabhakar-section}

Let \(\alpha,\beta,\gamma,q>0\) and \(z>0\) in the domain of convergence.  For the Prabhakar function
\[
        E_{\alpha,\beta}^{\gamma,q}(z)=\sum_{k\ge0}a_k z^k,
        \qquad
        a_k=\frac{(\gamma)_{qk}}{k!\,\Gamma(\alpha k+\beta)},
        \qquad
        (\gamma)_{qk}=\frac{\Gamma(\gamma+qk)}{\Gamma(\gamma)},
\]
the $z$-kernel \(K_z(k)=k/z\) is increasing  and the other two parameter kernels $L_\gamma(k)=\psi(\gamma+qk)-\psi(\gamma),$ and $L_\beta(k)=-\psi(\alpha k+\beta)$ are, respectively, increasing and decreasing.  The terms \(a_k z^k\) are smooth in \(z,\gamma,\beta\), and on every compact parameter set the relevant first and mixed derivatives are dominated by a summable majorant inside the disc of convergence, so Assumptions~\ref{asm:regularity} and~\ref{asm:mixed-regularity} hold.  The two-kernel Sign Rule yields
\[
        \partial_z\partial_\gamma\log E_{\alpha,\beta}^{\gamma,q}(z)\ge0,
        \qquad
        \partial_z\partial_\beta\log E_{\alpha,\beta}^{\gamma,q}(z)\le0 .
\]
On the other hand, for the one-kernel scenario, the weight ratios are ${(\gamma_1)}_{qk}/{(\gamma_2)}_{qk}$ and $\Gamma(\alpha k+\beta_2)/\Gamma(\alpha k+\beta_1)$, and both are monotone in \(k\). Therefore, the Sign Rule gives
\[
        \sgn(\gamma_1-\gamma_2)\,\partial_z
        \frac{E_{\alpha,\beta}^{\gamma_1,q}(z)}
             {E_{\alpha,\beta}^{\gamma_2,q}(z)}
        \ge0,
        \qquad
        \sgn(\beta_2-\beta_1)\,\partial_z
        \frac{E_{\alpha,\beta_1}^{\gamma,q}(z)}
             {E_{\alpha,\beta_2}^{\gamma,q}(z)}
        \ge0.
\]
We shall also use the following Fox--Wright coefficient fact.  It is stated separately because the Prabhakar coefficient is the one-numerator, one-denominator case of the Fox--Wright coefficient, up to a constant factor.

\begin{lemma}
\label{lem:foxwright-coeff-logconcave}
Let \(0<A\le1\), \(a\ge A\), and \(b,B>0\).  The sequence
\[
        u_k=\frac{\Gamma(a+Ak)}{\Gamma(k+1)\Gamma(b+Bk)},
        \qquad k\ge0,
\]
is log-concave.
\end{lemma}
\begin{proof}
For \(k\ge1\), set $\delta_{c,s}(k) = \log\Gamma(c+s(k+1))+\log\Gamma(c+s(k-1))-2\log\Gamma(c+sk).$ It follows that
\[
        \log u_{k+1}+\log u_{k-1}-2\log u_k
        =
        \delta_{a,A}(k)-\delta_{1,1}(k)-\delta_{b,B}(k).
\]
Moreover,
\[
\delta_{c,s}(k)
=
\int_0^1\!\int_0^1
s^2\psi'\!\bigl(c+s(k-1+u+v)\bigr)\,\dx[u]\,\dx[v] .
\]
Now, the digamma function derivative is $\psi'(x)=\sum_{m\ge0}{(x+m)^{-2}},$ so for \(t\ge0\) we have
\[
        A^2\psi'(a+At)
        =
        \sum_{m\ge0}
        \frac{1}{\bigl(a/A+t+m/A\bigr)^2}
        \le
        \sum_{m\ge0}\frac{1}{(1+t+m)^2}
        =
        \psi'(1+t),
\]
because \(0<A\le1\) and \(a/A\ge1\).  Hence, \(\delta_{a,A}(k)\le\delta_{1,1}(k)\), while \(\delta_{b,B}(k)\ge0\).  Consequently,  \(\log u_{k+1}+\log u_{k-1}-2\log u_k\le0\).
\end{proof}
Next, we turn to a question posed by  \cite[Section 7]{MehrezSitnik2018turan} regarding the remainders
\[
        R_n(z)=E_{\alpha,\beta}^{\gamma,q,n}(z)
        \coloneqq
        \sum_{k\ge n+1}a_k z^k,
\]
and whether or not they satisfy the Tur\'an-type inequality $R_n(z)R_{n+2}(z)\le R_{n+1}(z)^2$ for all admissible parameters?

The question is settled once more by the Sign Rule from the monotone likelihood ratio of laws induced by truncated series, where we consider each \(R_n\) as a  mass. The following result secures a positive range and, for \(q=1\), shows it is sharp as a uniform statement.

\begin{thm}
\label{thm:prabhakar-remainders}
Let \(\alpha,\beta>0\), \(z>0\), \(0<q\le1\), and \(\gamma\ge q\).  For every \(n\ge0\), it holds that
\[
        E_{\alpha,\beta}^{\gamma,q,n}(z)
        E_{\alpha,\beta}^{\gamma,q,n+2}(z)
        \le
        \left(E_{\alpha,\beta}^{\gamma,q,n+1}(z)\right)^2.
\]
For \(q=1\), the condition \(\gamma\ge1\) cannot be relaxed uniformly in \(\alpha,\beta\).
\end{thm}
\begin{proof}
Set \(b_k=a_kz^k\)  and introduce
\[
        T_m=\sum_{k\ge m}b_k=\Phi^+(m),
\]
the upper truncation of~\eqref{def.truncated_laws}. This means that \(R_n=T_{n+1}\).  By Lemma~\ref{lem:discrete-tail-sign-rule}, it is enough to show that \((b_k)_k\) is log-concave.  By Lemma~\ref{lem:foxwright-coeff-logconcave}, with \((a,A,b,B)=(\gamma,q,\beta,\alpha)\), the coefficient sequence \((a_k)\) is log-concave.  Since \((z^k)_k\) satisfies the log-concavity inequality with equality, \((b_k)_k\) is log-concave.  Hence,
\[
        T_{m+1}^2\ge T_mT_{m+2}.
\]
Taking \(m=n+1\) yields the stated remainder Tur\'an inequality.

For sharpness in the case \(q=1\), suppose \(0<\gamma<1\).  Take \(n=0\), where the Tur\'an inequality reads \(R_0R_2\le R_1^2\).  From \(R_0=a_1z+O(z^2)\), \(R_1=a_2z^2+O(z^3)\), and \(R_2=a_3z^3+O(z^4)\),
\[
        R_0R_2-R_1^2
        =
        a_1a_3z^4-a_2^2z^4+O(z^5)
        =
        \bigl(a_1a_3-a_2^2\bigr)z^4+O(z^5),
\]
so the sign of \(R_0R_2-R_1^2\) for small \(z>0\) is that of \(a_1a_3-a_2^2\). As \(\alpha\downarrow0\),
\[
        \frac{a_1a_3}{a_2^2}
        \longrightarrow
        \frac{2(\gamma+2)}{3(\gamma+1)}>1 .
\]
Therefore, for sufficiently small \(\alpha>0\), \(a_1a_3>a_2^2\), and the Tur\'an inequality fails for all sufficiently small positive \(z\).  Hence \(\gamma\ge1\) is necessary for a uniform \(q=1\) result.
\end{proof}

\subsection{A local threshold in the zero-balanced Gauss quotient}
\label{sec:baricz}

The covariance plus expectation identity recovers the threshold at which a monotonicity conjectured by Baricz~\cite[Conjecture~A]{Baricz2007MathZ} for the zero-balanced Gauss function fails, and names the effect that produces it. For \(a,b>0\) and \(x\in(0,1)\), put
\[
        R(x)=\frac{\HypF21(a+1,b+1;a+b+1;x)}{\HypF21(a,b;a+b;x)} .
\]
Baricz conjectured \(\partial_b R<0\) for all \(a,b>0\), \(x\in(0,1)\). Recently,  Qiu, Ma and Xiang~\cite{QiuMaXiang2024} disproved this by showing the conjecture failing once \(b>\sqrt{a(a+1)}\).  Our framework achieves the same result with a short proof that singles out the competing term
\begin{thm}[\cite{QiuMaXiang2024}]
\label{thm:baricz-threshold}
For every \(a>0\), the inequality \(\partial_b R(x)<0\) fails for all sufficiently small \(x>0\) whenever \(b>\sqrt{a(a+1)}\).
\end{thm}
\begin{proof}
Write \(\HypF21(a,b;a+b;x)=\sum_n w_n(b)x^n\), \(w_n=(a)_n(b)_n/((a+b)_n n!)>0\), and let \(N\sim P_x\) have the induced law \(P_x(n)\propto w_n x^n\) on \(\Nzero\).  The contiguous quotient is the mean \(R(x)=\E\,\ell(N)\) of the weight ratio \(\ell(n)=(1+n/a)(1+n/b)/(1+n/(a+b))\), increasing in \(n\).  Here \(b\) is the monotonicity parameter, with kernel \(K_b(n)=\partial_b\log w_n=\sum_{j=0}^{n-1}\bigl[(b+j)^{-1}-(a+b+j)^{-1}\bigr]\), increasing in \(n\).  The terms \(w_n(b)x^n\) are smooth in \(b\) and dominated on compact \(b\)-intervals by a summable majorant inside the disc of convergence, so Assumption~\ref{asm:regularity} holds.  The weight ratio \(\ell\) depends on \(b\) as well, so the covariance plus expectation identity (Proposition~\ref{lem:induced-expectation-derivative}) applies and gives
\[
        \partial_b R=\E\,\partial_b \ell(N)+\Cov\bigl(\ell(N),K_b(N)\bigr) .
\]
The first term \(\E\,\partial_b \ell(N)\) is \(\le0\), since \(\partial_b \ell(n)=-n(a+n)(a+2b+n)/[b^2(a+b+n)^2]\le0\). This fact alone would yield Baricz's inequality.  The covariance is positive by the Chebyshev integral inequality, \(\ell\) and \(K_b\) both increasing.  The conjecture therefore holds only while \(\E\,\partial_b \ell(N)\) dominates the covariance.  At leading order in \(x\), the two-point law \(P_x(0)=1+O(x)\), \(P_x(1)=\tfrac{ab}{a+b}\,x+O(x^2)\) gives
\[
        \partial_b R=\frac{a\,\bigl(b^2-a(a+1)\bigr)}{(a+b)^2(a+b+1)^2}\,x+O(x^2),
\]
which is negative for \(b<\sqrt{a(a+1)}\) and positive for \(b>\sqrt{a(a+1)}\). Hence the conjectured sign fails for \(b>\sqrt{a(a+1)}\).
\end{proof}

The covariance plus expectation identity explains the local failure mechanism: the negative term \(\E\,\partial_b\ell(N)\) competes with the positive covariance \(\Cov(\ell(N),K_b(N))\).  

\subsection{A shifted hypergeometric transform criterion}
\label{sec:kk-conj}

The shifted hypergeometric Tur\'anian studied by Kalmykov and Karp
\cite{KalmykovKarp2017} is most naturally handled by applying the variance rule
to a positive transform representation of the function itself.  The following single hypothesis furnishes a criterion for every rank \(p\ge q\): that the gamma ratio is a positive Laplace transform in \(\mu\).  For \(\mathbf a\in(0,\infty)^p\), \(\mathbf b\in(0,\infty)^q\) with \(p\ge q\), put
\[
        f(\mu;x)=\sum_{n\ge0}
        \frac{\Gamma(\mathbf a+\mu+n)}{\Gamma(\mathbf b+\mu+n)}
        \frac{x^n}{n!},
        \qquad \Gamma(\mathbf a+z)=\prod_i\Gamma(a_i+z).
\]

\begin{assertion}
\label{prop:kk-transform-criterion}
Suppose the gamma ratio admits a positive Laplace transform in \(\mu\),
\begin{equation}
\label{eq:kk-gamma-laplace}
        \frac{\Gamma(\mathbf a+\mu)}{\Gamma(\mathbf b+\mu)}
        =
        \int_{\mathbb R} e^{-\mu t}g_0(t)\,\dx[t],
        \qquad g_0(t)\ge0,\quad \mu>0 .
\end{equation}
Setting \(g_n(t)=e^{-nt}g_0(t)\) then represents the shifted ratio \(\Gamma(\mathbf a+\mu+n)/\Gamma(\mathbf b+\mu+n)=\int_{\mathbb R}e^{-\mu t}g_n(t)\,\dx[t]\) for every \(n\in\Nzero\). Then, for every \(x\ge0\) in the convergence domain, \(\mu\mapsto f(\mu;x)\) is log-convex on \((0,\infty)\), so
\[
        \Delta_f(\mu;\alpha,\beta;x)
        \coloneqq
        f(\mu+\alpha;x)f(\mu+\beta;x)
        -f(\mu;x)f(\mu+\alpha+\beta;x)
        \le0
\]
for all \(\mu>0\) and \(\alpha,\beta\ge0\).
\end{assertion}

\begin{proof}
Since \(x\ge0\), Tonelli's theorem permits exchanging the defining series for
\(f(\mu;x)\) with the positive integral representations of the shifted gamma
ratios, giving
\[
        f(\mu;x)=\int_{\mathbb R} e^{-\mu t}H_x(t)\,\dx[t],
        \qquad
        H_x(t)=\sum_{n\ge0}g_n(t)\frac{x^n}{n!}\ge0 .
\]
Thus \(f\) is a positive Laplace transform in \(\mu\), and Lemma~\ref{lem:positive-transform} gives log-convexity in \(\mu\) (and the variance identity \(\partial_\mu^2\log f=\Var(X)\ge0\) wherever the second moment is finite).  This yields the required Tur\'an inequality.
\end{proof}

Only the positive transform representation~\eqref{eq:kk-gamma-laplace} depends on the rank, and it is secured by the Karp--Prilepkina representation, the same positive representation mechanism as the Stieltjes quotient above.  In the balanced case \(p=q\), the ratio is the Mellin transform of the Meijer-\(G\) density \(\rho_{\mathbf a,\mathbf b}(s)=G^{q,0}_{q,q}\!\left(s\,\middle|\begin{smallmatrix}\mathbf b\\ \mathbf a\end{smallmatrix}\right)\) on \((0,1)\).  The common shift \(n\) replaces this density by \(s^n\rho_{\mathbf a,\mathbf b}(s)\).  After \(s=e^{-t}\), this gives a positive representing density \(g_n\) in~\eqref{eq:kk-gamma-laplace}, supported on \((0,\infty)\).  A standard sufficient condition for \(\rho_{\mathbf a,\mathbf b}\ge0\) is the Karp--Prilepkina condition \(\sum_{i}(t^{a_i}-t^{b_i})\ge0\) on \((0,1)\), for example the ordered partial-sum condition \(\sum_{i\le k}a_i\le\sum_{i\le k}b_i\) after arranging both vectors increasingly.

In the unbalanced case \(p>q\), a sufficient condition is the existence of a \(q\)-subvector \(\mathbf a'\subseteq\mathbf a\) satisfying the positivity condition relative to \(\mathbf b\), so that \(\Gamma(\mathbf a'+\mu)/\Gamma(\mathbf b+\mu)\) has a positive Laplace representation supported on \((0,\infty)\).  If the remaining numerator factors also have positive Laplace representations, then their product is represented by the convolution of positive measures, and~\eqref{eq:kk-gamma-laplace} follows. Outside such a positive representing range, this proof no longer supplies a positive transform, and the variance argument does not apply directly.

\section{Conclusion}
\label{sec:conclusion-final}

The paper develops a sign rule for positive series and integral transforms. The examples illustrate how the common probabilistic mechanism replaces several separate special-function arguments.  It yields parameter and argument monotonicity for generalised hypergeometric quotients.  It recovers and extends the Karp--Sitnik monotonicity of Stieltjes-transform quotients, removing the restriction to \(\sigma\ge1\) and adding a companion upper bound.  It settles a Prabhakar remainder Tur\'an inequality through the discrete tail rule, sharp at \(q=1\).  It locates the threshold of Baricz's zero-balanced Gauss problem and names the covariance as the competing effect.  And it furnishes a rank-uniform positive-transform criterion for shifted hypergeometric Tur\'anians of Kalmykov--Karp type. However, the limitation is also clear: the method requires a positive representation.  If the available series or integral is signed, as for alternating Bessel expansions or the usual hypergeometric forms of some orthogonal polynomials, the probability-law argument applies only after finding a different positive representation.

\printbibliography

\end{document}